\documentclass[letterpaper,12pt,twoside]{amsart}
\usepackage[left=2.5cm, right=2.5cm,tmargin = 2.5cm]{geometry}
\usepackage{amsmath,amssymb,amsfonts,amsthm}

\usepackage[all]{xy}
\usepackage{tikz}
\usepackage{color}
\usepackage{graphicx}

\usepackage{indentfirst}
\usepackage[utf8]{inputenc}
\usepackage{lastpage}
\usepackage{enumitem}
\usepackage{multicol}

\renewcommand{\title}[1]{\begin{center}
    \begin{minipage}[t]{125mm}
        \Large\bf\begin{center} #1
        \end{center}
    \end{minipage}
\end{center}
\vskip3mm }

\renewcommand{\author}[2]{\centerline{#1}
\vskip2mm
\par
\centerline{\small\it
 \begin{minipage}[t]{145mm}
  \begin{center}
 #2
  \end{center}
 \end{minipage}
}
\vspace{8mm}
}

\newcommand{\resumen}[1]{
\begin{center}
\textbf{Abstract}\\

\begin{minipage}[t]{130mm}
\small
#1
\end{minipage}
\end{center}
}
\newcommand{\palabrasclave}[1]{
\begin{center}
\begin{minipage}[t]{130mm}
\small
\emph{Keywords: }
#1
\end{minipage}
\end{center}
}

\newtheorem{teorema}{Theorem}[section]
\newtheorem{lema}{Lemma}[section]

\newtheorem{proposicion}{Proposition}[section]

\theoremstyle{definition}

\newtheorem{ejemplo}{Example}[section]

\theoremstyle{remark}
\newtheorem{obs}{Observation}[section]

\newcommand\R{\mathbb{R}}

\usepackage{mathrsfs}

\begin{document}
\setcounter{page}{1}

\normalsize

\title{ \vspace{0.8in} Controllability of induced bilinear systems on the sphere}

\author{Marco A. Colque-Choquecallata$^1$, Efrain Cruz-Mullisaca$^1$, Victor H. Patty-Yujra$^{1}$}%
{
$^1$ Instituto de investigación Matemática, Universidad Mayor de San Andrés, Bolivia 
}

\newcommand{\Addresses}{{
  \bigskip
  \footnotesize

\begin{itemize}[label={},itemindent=-2em,leftmargin=2em]
\item \textsc{ M.Colque-Choquecallata}, Instituto de Investigación Matemática, Universidad Mayor de San Andrés.\\  
    Calle 27 de Cota Cota, Campus Universitario, Edificio de Ciencias Puras y Naturales, 1er Piso. La Paz-- Bolivia\par\nopagebreak
    \textit{E-mail:}\texttt{marco.colque@gmail.com }
\item \textsc{ E.Cruz-Mullisaca}, Instituto de Investigación Matemática, Universidad Mayor de San Andrés.\\  
    Calle 27 de Cota Cota, Campus Universitario, Edificio de Ciencias Puras y Naturales, 1er Piso. La Paz-- Bolivia\par\nopagebreak
    \textit{E-mail:}\texttt{ecruz@fcpn.edu.bo}
\item \textsc{ V.Patty-Yujra}, Instituto de Investigación Matemática, Universidad Mayor de San Andrés.\\  
    Calle 27 de Cota Cota, Campus Universitario, Edificio de Ciencias Puras y Naturales, 1er Piso. La Paz-- Bolivia\par\nopagebreak
    \textit{E-mail:}\texttt{vpattyy@fcpn.edu.bo}
\end{itemize}
}}

\hrule
\resumen{
In this paper, we investigate the controllability of bilinear control systems of the form $\dot{s} = As + uBs$, where $s \in \mathbb{S}^2$ and $A, B \in gl(3, \mathbb{R})$ are skew-symmetric matrices. First, we prove that the algebraic condition $[A, B] \neq 0$ ensures that the Lie algebra rank condition is satisfied for these systems. Next, we show that this same condition implies the controllability of the system. Finally, in an explicit and descriptive manner, we demonstrate controllability by exhibiting trajectories that transfer a given initial state to another.}

\palabrasclave{
Controllability; Bilinear systems; Lie algebra rank condition.}

\hrule

\section{Introduction}
Consider a bilinear control system in $\R^3,$ with control $u\in\mathbb{R},$ given by
\begin{equation}\label{sb}
\Sigma: \hspace{0.3in} \dot{x}=Ax+uBx, \hspace{0.3in} x\in\R^3\smallsetminus\{0\}, 
\end{equation} where $A, B\in gl(3,\R),$ are $3\times 3$ matrices with real coefficients. The system $\Sigma$ can be projected radially onto the unit sphere $\mathbb{S}^2\subset\mathbb{R}^3:$ for $x(t)\in\R^3\smallsetminus\{0\},$ solution of \eqref{sb}, writing $
s(t):=x(t)/|x(t)| \in \mathbb{S}^2$ we get the {\it induced bilinear system} on the sphere $\mathbb{S}^2$ given by
\begin{equation}\label{sbi}
 \dot{s}=h_A(s)+uh_B(s), \hspace{0.2in} s\in \mathbb{S}^2, u\in \R,
\end{equation}
where, for instance, $h_A(s)= As -\left\langle  As,s \right\rangle s,$ for all $s\in\mathbb{S}^2.$ If $A$ and $B$ are skew-symmetric, we have $h_A(s)=As$ and $h_B(s)=Bs,$ for all $s\in \mathbb{S}^2,$ thus, the induced bilinear system \eqref{sbi} remains as 
\begin{equation}\label{sbi1}
\mathbb{P}\Sigma: \hspace{0.3in} \dot{s}=As+uBs, \hspace{0.2in} s\in \mathbb{S}^2, u\in \R.
\end{equation}

The main goal of this paper is to study the controllability problem of the induced bilinear system $\mathbb{P}\Sigma$ on $\mathbb{S}^2.$  Using techniques from Lie Theory, it is shown that the control system \eqref{sbi1} is controllable on $\mathbb{S}^2,$ see \cite{B,V}; however, our approach in this paper is descriptive, that is, we will solve the controllability problem by making explicit the trajectories of the system that transfer a certain initial state to another.

We quote the following related papers: in \cite[Corollary 12.2.6]{CK} it is shown that the controllability of \eqref{sbi} is a necessary condition for the controllability of \eqref{sb} on $\R^3\smallsetminus\{0\},$ this motivates the present paper. The controllability problem of bilinear systems in the plane was studied by the second author \emph{et al} in \cite{C}, however, in this work the study of the Lie algebra rank condition of the system was not contemplated; the second and third author in \cite{CP} obtain an algebraic condition that ensures the range condition and with elementary techniques recover the main result of \cite{C}.

The outline of the paper is as follows: in Section \ref{sec2} we will study the Lie algebra rank condition for the induced bilinear system \eqref{sbi} on $\mathbb{S}^2$ (Theorem \ref{thlarc}); in particular, in Theorem \ref{thlarcant} we give an algebraic condition that ensures the range condition for the system $\mathbb{P}\Sigma.$  In Section \ref{sec3} we study the controllability problem for the induced bilinear system $\mathbb{P}\Sigma:$ we will prove in Theorem \ref{teocontrol} that the hypothesis that ensures the range condition also implies the controllability of $\mathbb{P}\Sigma.$ Finally, we give an illustrative example of the technique used in the proof of Theorem \ref{teocontrol}. 

\section{Lie algebra rank condition for induced bilinear systems}\label{sec2}

It is well known that a necessary condition for the controllability of the induced bilinear system  \eqref{sbi} on $\mathbb{S}^2$ is given by the {\it Lie algebra rank condition} (LARC): if we denote by $\chi(\mathbb{S}^2)$ the Lie algebra of differentiable vector fields on the sphere $\mathbb{S}^2$ then, the Lie algebra generated by the system, given by $$\mathcal{L}_{\Sigma}:=\langle h_A +uh_B\mid u\in \mathbb{R}  \rangle \subset \chi(\mathbb{S}^2),$$
is such that for all $s\in\mathbb{S}^2,$ the vector space $\mathcal{L}_{\Sigma}(s) \subset T_s\mathbb{S}^2$ satisfies
\begin{equation}
\dim \mathcal{L}_{\Sigma}(s)=\dim\langle h_A(s)+uh_B(s) \mid u\in\mathbb{R} \rangle =2;
\end{equation} see for example \cite{CK, E}. 

In this section, we will study the Lie algebra rank condition for the induced bilinear system \eqref{sbi} on $\mathbb{S}^2;$ in particular, the rank condition for the system \eqref{sbi1}. We need the following lemmas.

\begin{lema}\label{ceros} The zeros of the induced vector field $h_A$ on $\mathbb{S}^2,$ with $A\in gl(3,\mathbb{R}),$ follows the direction of the eigenvectors of the matrix $A.$
\end{lema}
\begin{proof} Using the Lagrange identity we have $|h_A(s)|=|As \times s|,$ for all $s\in\mathbb{S}^2,$ where $\times$ denotes the usual cross product in $\R^3.$ Thus,  $h_A(s)=0$ if and only if  $As=\lambda s,$ for some $\lambda\in\R.$
\end{proof}

We consider the usual Lie bracket of matrices $[A,B]=AB-BA,$ for $A$ and $B$ in $gl(3,\R).$ We have the following interesting relation.

\begin{lema}\label{clie} For all $A$ and $B$ in $gl(3,\R)$ we have
\begin{equation*}
h_{[A,B]}=[h_A,h_B],
\end{equation*} where $[h_A,h_B]=dh_A(h_B)-dh_B(h_A)$ denotes the Lie bracket on $\chi(\mathbb{S}^2).$
\end{lema}
\begin{proof}
By a straightforward computation, for all $s\in\mathbb{S}^2$ and $v\in T_s\mathbb{S}^2,$ we have
\begin{equation*}
d(h_A)_s(v)=Av-\langle Av,s\rangle s -\langle As,v\rangle s-\langle As,s\rangle v;
\end{equation*} therefore $d(h_A)_s(h_B(s))-d(h_B)_s(h_A(s)) 
= (AB-BA)s-\langle (AB-BA)s,s\rangle s= h_{[A,B]}$
like we wanted.
\end{proof}

\begin{proposicion}\label{liealgebra}
The map $\Phi : gl(3,\R) \longrightarrow \chi(\mathbb{S}^2),$ given by $\Phi(A) = h_A$
is a homomorphism of Lie algebras. In particular, if $Id\in gl(3,\R)$ denotes the identity matrix, we have $h_{Id}=0.$
\end{proposicion}
\begin{proof}
For all $s\in\mathbb{S}^2$ we have
$h_{A+uB}(s)=h_A(s)+uh_B(s)$ and, 
using Lemma \ref{clie} we get
$\Phi [A,B]=h_{[A,B]}=[h_A,h_B]=[\Phi(A), \Phi(B)].$
\end{proof}

\begin{proposicion}\label{lineald} Let $A$ and $B$ be in $gl(3,\R),$ $x\in\R^3\smallsetminus\{0\}$ and write $s:=x/|x|.$ Then, $\{ x, Ax, Bx \} \subset \R^3$ is linearly independent if and only if  $\left\lbrace h_A\left( s\right), h_B\left( s\right) \right\rbrace  \subset  T_s\mathbb{S}^2$ is linearly independent. 
\end{proposicion}
\begin{proof} We suppose that $\{ x, Ax, Bx \}\subset \R^3$ is linearly independent, without loss of generality, there exist $\alpha, \beta$ such that $Bx=(\alpha Id  + \beta A)x,$ thus
$h_B(s) = \alpha h_{Id}(s)+\beta h_A(s)=\beta h_A(s),$ {\it i.e.,} $h_A(s)$ and $h_B(s)$ are linearly dependent in $T_s\mathbb{S}^2.$ Reciprocally, if $\left\lbrace h_A\left( s\right), h_B\left( s\right) \right\rbrace \subset T_s\mathbb{S}^2$ is linearly dependent, {\it i.e.,} if the exists $\lambda$ such that $h_A(s)=\lambda h_B(s),$ we get $h_{A-\lambda B} (s)=0;$ using Lemma \ref{ceros}, there exists $\gamma$ with $(A-\lambda B)s=\gamma s,$ therefore $As=\gamma s+\lambda Bs,$ {\it i.e.,} $s, As$ and $Bs$ are linearly dependent.
\end{proof}

\begin{obs}\label{ldepen1} If $x_0\in\mathbb{R}^3\smallsetminus\{0\}$ is a eigenvector of $A$ or $B,$ the set $\{x_0, Ax_0, Bx_0\}$ is linearly dependent; according to Proposition \ref{lineald}, the set $\{ h_A(s), h_B(s) \}  \subset  T_s\mathbb{S}^2$ is linearly dependent in $s=x_0/|x_0|\in\mathbb{S}^2;$ in this case, we note that $h_A(s)=0$ or $h_B(s)=0,$ thus, it is necessary to consider an additional direction in the vector space $T_s\mathbb{S}^2.$\end{obs}

\begin{lema}\label{ld0} For all $s\in\mathbb{S}^2,$ the vector space $\mathcal{L}_{\Sigma}(s) \subset T_s\mathbb{S}^2$ it is generated by the set \begin{equation*}
\mathcal{H}:=\left\lbrace h_A(s),\ h_B(s),\ h_{[A,B]}(s),\ h_{[A,[A,B]]}(s),\ h_{[B,[A,B]]}(s), \ldots  \right\rbrace.
\end{equation*}
\end{lema}
\begin{proof}
Since $\mathcal{L}_{\Sigma}(s) \subset T_s\mathbb{S}^2$ is generated by the vectors $Z_u(s)=h_A(s)+uh_B(s)\in T_s\mathbb{S}^2,$ with $u\in \R,$ taking $u=0,$ we have the direction $h_A(s);$ moreover, for all $u\neq v$ in $\R,$ we have the direction $Z_u(s)-Z_v(s)=(u-v)h_B(s).$ On the other hand, using Proposition \ref{liealgebra}, the Lie bracket of $Z_u=h_A+uh_B$ and $Z_v=h_A+vh_B,$ in $\mathcal{L}_{\Sigma},$ satisfies
\begin{align*}
[Z_u,Z_v] = [h_A+uh_B,h_A+vh_B] 
=(v-u)[h_A,h_B]=(v-u)h_{[A,B]};
\end{align*} thus, the Lie bracket of $Z_u$ and $Z_v,$ in $s,$ follows the direction of $h_{[A,B]}(s).$ Therefore, all linear combinations of $Z_u,$ with $u\in\R,$ and the Lie brackets of them, are generated by $\mathcal{H}.$
\end{proof}

According to Lemma \ref{ld0}, we have the following result.

\begin{teorema}\label{thlarc} The induced bilinear system \eqref{sbi} satisfies the Lie algebra rank condition (LARC) on $\mathbb{S}^2$ if $$\mbox{rank} \begin{pmatrix}
h_A(s) & h_B(s) & h_{[A,B]}(s) 
\end{pmatrix} = 2,
$$ for all $s\in\mathbb{S}^2.$ In particular, if $[A,B]=0,$ the bilinear system \eqref{sbi} does not satisfy LARC.
\end{teorema}
\begin{proof}
The condition ensures the linear independence of two vectors in $T_s\mathbb{S}^2,$ for all $s\in\mathbb{S}^2.$ If $[A,B]=0,$ we have $h_{[A,B]}=h_{[A,[A,B]]}=h_{[B,[A,B]]}=\cdots=0;$ by Observation \ref{ldepen1}, the system \eqref{sbi} does not satisfy LARC.   
\end{proof}

\begin{ejemplo} We consider the bilinear control system $\Sigma$ in $\R^3\smallsetminus\{0\}$ given by
\begin{equation*}
\dot{x}=Ax+uBx=\begin{pmatrix}
0 & 1 & 0 \\ -1 & 0 & 0 \\ 0 & 0 & 0 
\end{pmatrix}x+u \begin{pmatrix}
0 & 0 & 0 \\ 0 & 0 & 1 \\ 0 & -1 & 0
\end{pmatrix}x. 
\end{equation*} Writing $s=\begin{pmatrix}
s_1 & s_2 & s_3
\end{pmatrix}^{\top}\in \mathbb{S}^2,$ by a straightforward computation, the induced bilinear system is given by
\begin{equation*}
\dot{s}=h_A(s)+uh_B(s)=\begin{pmatrix}
s_2 \\ -s_1 \\ 0
\end{pmatrix} + u \begin{pmatrix}
0 \\ s_3 \\ -s_2
\end{pmatrix}, \ \ \ s\in\mathbb{S}^2, \ \ u\in\R.
\end{equation*} 
Since $
[A,B]=AB-BA=\begin{pmatrix}
0 & 0 & 1 \\ 0 & 0 & 0 \\ -1 & 0 & 0 
\end{pmatrix},$ we have
$h_{[A,B]}(s)=\begin{pmatrix}
s_3 \\ 0 \\ -s_1
\end{pmatrix}.$
Finally, since $s\in\mathbb{S}^2,$ without loss of generality, we can assume that $s_2\neq 0,$ thus
\begin{align*}
\mbox{rank} \begin{pmatrix}
h_A(s) & h_B(s) & h_{[A,B]}(s) 
\end{pmatrix} =\mbox{rank} \begin{pmatrix}
s_2 & 0 & s_3 \\ -s_1 & s_3 & 0 \\ 0 & -s_2 & -s_1
\end{pmatrix} = \mbox{rank} \begin{pmatrix}
1 & 0 & \frac{s_3}{s_2} \\ 0 & 1 & \frac{s_1}{s_2} \\ 0 & 0 & 0
\end{pmatrix}=2;
\end{align*} according to Theorem \ref{thlarc}, the induced bilinear system satisfies the rank condition on $\mathbb{S}^2.$
\end{ejemplo}

\begin{ejemplo} We consider the bilinear control system $\Sigma$ in $\R^3\smallsetminus\{0\}$ given by
\begin{equation*}
\dot{x}=Ax+uBx=\begin{pmatrix}
0 & 1 & 0 \\ 0 & 0 & 0 \\ 0 & 0 & 0 
\end{pmatrix}x+u \begin{pmatrix}
0 & 0 & 1 \\ 0 & 0 & 0 \\ 0 & 0 & 0
\end{pmatrix}x. 
\end{equation*} Writing as before $s=\begin{pmatrix}
s_1 & s_2 & s_3
\end{pmatrix}^{\top}\in \mathbb{S}^2,$ the induced bilinear system remains as
\begin{equation*}
\dot{s}=h_A(s)+uh_B(s)=\begin{pmatrix}
s_2 - s_1^2s_2 \\ -s_1s_2^2 \\ -s_1s_2s_3
\end{pmatrix} + u \begin{pmatrix}
s_3-s_1^2s_3 \\ -s_1s_2s_3 \\ -s_1s_3^2
\end{pmatrix}, \ \ \ s\in\mathbb{S}^2, \ \ u\in\R.
\end{equation*} 
Since $[A,B]=0,$ the induced system does not satisfy LARC: we have $h_A(s)=h_B(s)=0$ in $s=(1,0,0) \in\mathbb{S}^2.$ 
\end{ejemplo}

We will prove that if the matrices $A$ and $B$ are skew-symmetric, the induced bilinear system $\mathbb{P}\Sigma$ given in \eqref{sbi1} satisfies LARC on $\mathbb{S}^2 ;$ we need the following lemma.

\begin{lema}\label{jordanA} If $A\in gl(3,\R)$ is skew-symmetric, there exist $a>0$ and $P\in gl(3,\R)$ orthogonal such that
\begin{equation*}
\mathcal{J}(A):=P^{-1}AP=\begin{pmatrix}
0 & a & 0 \\-a & 0 & 0\\ 0 & 0 & 0
\end{pmatrix}.  
\end{equation*} 
\end{lema}
\begin{proof}
We suppose that $ A=\begin{pmatrix}
    0&a_1&a_2\\-a_1 &0&a_3\\ -a_2&-a_3&0
\end{pmatrix}\neq 0.$ If $a_2=0$ and $a_3=0,$ taking $P=Id$ we get the result; we can suppose that $\alpha:=a_2^2+a_3^2\neq 0.$ Writing $a:=\sqrt{a_1^2+a_2^2+a_3^2},$ the characteristic polynomial of $A$ is given by
$p(\lambda)=\det(\lambda I-A)=\lambda(\lambda^2+a^2),$ thus, $\lambda_{1,2}= ia$ and $\lambda_3=0$ are its eigenvalues.
The associated eigenvectors are given by
$$ v_1=\dfrac{1}{a\sqrt{\alpha}}\begin{pmatrix}
  -a_1 a_3\\ a_1 a_2\\ \alpha
\end{pmatrix},\quad v_2=\dfrac{-1}{\sqrt{\alpha}}\begin{pmatrix}
    a_2\\ a_3\\0
\end{pmatrix},\quad v_3=\dfrac{1}{a}\begin{pmatrix}
     a_3 \\-a_2 \\ a_1
\end{pmatrix},$$
which form an orthonormal set. The matrix $P:=\begin{pmatrix} v_1 & v_2 & v_3  
\end{pmatrix}$ is orthogonal and satisfies the equality described above
\end{proof}

\begin{obs}\label{obcasosimetrico}
We consider the bilinear control system $\Sigma,$ given in \eqref{sb}, with $A$ and $B$ skew-symmetric, and the matrix $P$ given in Lemma \ref{jordanA}. Writing $y=P^{-1}x,$ we get
\begin{equation*}
\dot{y}=P^{-1}\dot{x}= P^{-1}(A+uB)x =P^{-1}(A+uB)Py=P^{-1}APy+uP^{-1}BPy;
\end{equation*} the matrix $\tilde{B}:=P^{-1}BP$ is skew-symmetric, thus, we have the equivalent bilinear control system
\begin{equation*}\label{sisbJordan}
    \dot{y}=\mathcal{J}(A)y+u\tilde{B}y. 
\end{equation*} Moreover, since 
\begin{equation*}
[\mathcal{J}(A),\tilde{B}]=[P^{-1}AP,P^{-1}BP]=P^{-1}[A,B]P,
\end{equation*} we have $[\mathcal{J}(A),\tilde{B}]\neq 0$ if and only if $[A,B]\neq 0.$ Therefore, it is sufficient to study the Lie algebra rank condition of induced bilinear systems of the form
\begin{equation}\label{sistemajordan}
\mathbb{P}\Sigma: \hspace{0.3in} \dot{s}=As+uBs=\begin{pmatrix}
    0&a&0\\-a&0&0\\0&0&0
\end{pmatrix}s+u \begin{pmatrix}
    0&b_1&b_2\\-b_1&0&b_3\\-b_2&-b_3&0
\end{pmatrix}s, \hspace{0.3in} a>0, \ s\in\mathbb{S}^2.
\end{equation}
Finally, in this case note that: $[A,B]\neq 0$ if and only if $b_2^2+b_3^2\neq 0.$
\end{obs}

\begin{teorema}\label{thlarcant} The induced bilinear system $\mathbb{P}\Sigma:$ $\dot{s}=As+uBs,$ given by \eqref{sistemajordan}, satisfies the Lie algebra rank condition on the sphere $\mathbb{S}^2$ if $[A,B]\neq 0.$
\end{teorema}
\begin{proof}
By a straightforward computation, we have
\begin{equation*}
[A, B]=\begin{pmatrix}
             0&0&ab_3\\ 0&0&-ab_2\\ -ab_3&ab_2&0
         \end{pmatrix};       
\end{equation*} thus, for all $s=(s_1,s_2,s_3)\in \mathbb{S}^2$ we get
\begin{equation*}
h_A(s)=\begin{pmatrix}
        as_2\\-as_1\\0
    \end{pmatrix}, \hspace{0.3in}  h_B(s)=\begin{pmatrix}
        b_1s_2+b_2 s_3\\-b_1s_1+b_3 s_3\\-b_2s_1-b_3s_2
    \end{pmatrix} \hspace{0.2in}\mbox{and}\hspace{0.2in}
    h_{[A,B]}(s)=\begin{pmatrix}
        ab_3 s_3\\-ab_2s_3\\-ab_3s_1+ab_2s_2
    \end{pmatrix}. 
\end{equation*}
We will show that the matrix
\begin{equation*}
C=\begin{pmatrix}
h_A(s) & h_B(s) & h_{[A,B]}(s) 
\end{pmatrix}
=\begin{pmatrix}
as_2&b_1s_2+b_2s_3&ab_3s_3\\ -as_1&-b_1s_1+b_3s_3&-ab_2s_3\\0&-b_2s_1-b_3s_2&-ab_3s_1+ab_2s_2
\end{pmatrix}    
\end{equation*}
has rank 2 at all points of $\mathbb{S}^2.$ Let us consider the following cases:

\noindent\textbf{Case 1.} If $s_2\neq 0$ and $s_3\neq 0,$ we easily obtain
\begin{equation*}
C \sim  \begin{pmatrix}
as_2&b_1s_2+b_2s_3&ab_3s_3\\ 0&b_2\dfrac{s_1 s_3}{s_2}+b_3s_3&-ab_2s_3+ab_3\dfrac{s_1s_3}{s_2}\\0&0&0
\end{pmatrix};
\end{equation*} thus, $C$ has rank 1 if and only if $b_2 s_1 s_3+b_3 s_2 s_3=0$ and $-ab_2 s_2 s_3+a b_3 s_1 s_3=0$, this happens if and only if $b_2=b_3=0,$ therefore, $\mbox{rank}(C)=2$ if and only if $b_2^2+b_3^2\neq 0.$

\noindent\textbf{Case 2.} If $s_2=0,$ we get
\begin{equation*}
C \sim \begin{pmatrix}
        -as_1&-b_1s_1+b_3s_3&-ab_2s_3\\0&b_2s_3&ab_3s_3\\0&-b_2s_1&-ab_3s_1
    \end{pmatrix}; 
\end{equation*} consider the following cases:
\begin{enumerate}
    \item[\bf 2.1.] If $s_1\neq 0$ and $s_3\neq 0,$ we have that $\mbox{rank}(C)=2$ if and only if $b_2^2+b_3^2\neq 0.$
    \item[\bf 2.2.] If $s_1=0,$ we have $s_3=1,$ therefore
\begin{equation*}
    C=\begin{pmatrix}
        0&b_2&ab_3\\0&b_3&-ab_2\\0&0&0
    \end{pmatrix};
\end{equation*} thus, we have that $\mbox{rank}(C)=2$ if and only if $b_2^2+b_3^2\neq 0.$

\item[\bf 2.3.] If $s_3=0,$ we have $s_1=1,$ thus  
\begin{equation*}
    C=\begin{pmatrix}
        0&0&0\\-a&-b_1&0\\0&-b_2&-ab_3
    \end{pmatrix};
\end{equation*} as before, $\mbox{rank}(C)=2$ if and only if $b_2^2+b_3^2\neq 0.$
\end{enumerate} 
\textbf{Case 3.} If $s_3=0$ we can assume that $s_2\neq 0;$ in fact, if $s_2=0$ we return to Case 2.3. We have thus \begin{equation*}
    C\sim \begin{pmatrix}
            as_2&b_1s_2&0\\ 0&0&0\\0&-b_2s_1-b_3s_2&-ab_3s_1+ab_2s_2
        \end{pmatrix};
\end{equation*} therefore,
$\mbox{rank}(C)=1$ if and only if $-b_2s_1-b_3s_2=0$ and $ab_2s_2-ab_3s_1=0,$ this happens if and only if $b_2=b_3=0.$ Finally, $\mbox{rank}(C)=2$ if and only if $b_2^2+b_3^2\neq 0.$
\end{proof}

\section{Controllability of induced bilinear systems }\label{sec3}
According to Observation \ref{obcasosimetrico}, in this section we will study the controllability problem of induced bilinear systems of the form
\begin{equation}\label{sistemajordan1}
\mathbb{P}\Sigma: \hspace{0.3in} \dot{s}=As+uBs=\begin{pmatrix}
    0&a&0\\-a&0&0\\0&0&0
\end{pmatrix}s+u \begin{pmatrix}
    0&b_1&b_2\\-b_1&0&b_3\\-b_2&-b_3&0
\end{pmatrix}s, \hspace{0.3in} a>0, \ s\in\mathbb{S}^2;
\end{equation} 
we will explicitly determine the trajectories and controls that solve the controllability problem. Writing $$\beta:=\sqrt{(a+ub_1)^2+(ub_2^2)+(ub_3)^2},$$ the characteristic polynomial of the matrix $$C=A+uB=\begin{pmatrix}
    0&a+ub_1&ub_2\\ -(a+ub_1)&0&ub_3\\-ub_2&-ub_3&0
\end{pmatrix}$$
is given by $p(\lambda)= \lambda\left[\lambda^2+\beta^2\right];$
 thus, the eigenvalues of $C$ are given by $\lambda_1=0$ y $\lambda_{2,3}=\pm i\beta.$ The eigenvectors associated are determined considering the following cases: \vspace{0.1in}

\noindent\textbf{Case 1.} If $u=0:$ 
\begin{itemize}
    \item For $\lambda_1=0$ we have
\begin{equation*}
  \lambda_1I-C=\begin{pmatrix}
        0&-a&0\\a&0&0\\0&0&0
    \end{pmatrix}\sim \begin{pmatrix}
        1&0&0\\0&1&0\\0&0&0
    \end{pmatrix},
\end{equation*}
thus, $v_1:=(0,0,1)$ is the eigenvector associated with $\lambda_1.$

\item For $\lambda_2=i\beta=ia,$ we have
\begin{equation*}
    \lambda_2I-C=\begin{pmatrix}
        ia&-a&0\\a&ia&0\\0&0&ia
    \end{pmatrix}\sim  \begin{pmatrix}
        ia&-a&0\\0&0&ia\\0&0&0
    \end{pmatrix}
\end{equation*}
thus, $v=:v_2+iv_3=(-i,1,0)$ is the eigenvector associated with $\lambda_{2,3}.$ For the conjugate eigenvalue $\lambda_3=-i\beta$ a similar result is obtained.
\end{itemize}

\noindent Thus, the trajectories of the system $\mathbb{P}\Sigma$, for control $u=0,$ are given by
\begin{equation*} 
    s(t)=c_1 e^{\lambda_1t}v_1^{T} +c_2\left[\cos(at)v_2^T -\sin(at)v_3^T\right]+c_3\left[ \cos(at)v_3^T +\sin(at)v_2^T\right], 
\end{equation*} or more precisely given by
 \begin{equation*}
     s(t)=\begin{pmatrix}
        c_2\sin(at)-c_3\cos(at)\\c_2 \cos(at)+c_3\sin(at)\\c_1
    \end{pmatrix},\quad t\in \mathbb{R}.
 \end{equation*}

\begin{obs}\label{solucion}
Given $s_0=(a_0,b_0,c_0)\in \mathbb{S}^2;$ the trajectory of $\mathbb{P}\Sigma$, with control $u=0$ and initial condition $s_0,$ is given by
\begin{equation*}
s(t;s_0,0)=\begin{pmatrix}
        b_0\sin(at)+a_0\cos(at)\\b_0 \cos(at)-a_0\sin(at)\\c_0
\end{pmatrix},
\end{equation*}
which are circles on the sphere $ \mathbb{S}^2$ with constant height $z=c_0.$
\end{obs}

\noindent\textbf{Case 2.} If $u\neq 0:$ 
\begin{itemize}
\item For $\lambda_1=0,$ we have
\begin{equation*}
     \lambda_1I-C\sim \begin{pmatrix}
        ub_2&ub_3&0\\0&a+ub_1&ub_2\\a+ub_1&0&-ub_3
    \end{pmatrix}\\
    \sim \begin{pmatrix}
        ub_2&ub_3&0\\0&a+ub_1&ub_2\\0&0&0
    \end{pmatrix}
\end{equation*}
thus, the eigenvectors are of the form $$\left(\dfrac{ub_3}{a+ub_1},-\dfrac{ub_2}{a+ub_1},1\right)t, \hspace{0.2in} t\in\mathbb{R}$$ when $a+ub_1\neq 0;$ taking $t=a+ub_1$ we have $v_1=(ub_3,-ub_2,a+ub_1).$ In the case $a+ub_1=0,$ a similar result is obtained.

\item For $\lambda_2=i\beta$ we have
\begin{equation*}
    \lambda_2 I- C\sim  \begin{pmatrix}
    i\beta &-(a+ub_1)& -ub_2\\0&i\dfrac{u^2\alpha}{\beta}&-ub_3-i\dfrac{(a+ub_1)ub_2}{\beta}\\0&0&0\end{pmatrix},
\end{equation*}
where $\alpha:=b_2^2+b_3^2;$ thus, if $\alpha\neq 0,$ the associated eigenvector to $\lambda_2$ is given by
$$v=v_2+iv_3=\left( -\dfrac{b_3(a+ub_1)}{u\alpha}-i\dfrac{b_2\beta}{u\alpha},\dfrac{b_2(a+ub_1)}{u\alpha}-i\dfrac{b_3\beta}{u\alpha}, 1 \right).$$ For the conjugate eigenvalue $\lambda_3=-i\beta$ a similar result is obtained.
\end{itemize}

\noindent Thus, the trajectories of the system $\mathbb{P}\Sigma$, for control $u\neq 0,$ are given by
$$s(t)=c_1 v_1^T+c_2(\cos(\beta t)v_2^T-\sin(\beta t)v_3^T)+c_3(\cos(\beta t)v_3^T+\sin(\beta t)v_2^T),$$
or more precisely given by
\begin{equation*}
 s(t)=\begin{pmatrix}
 c_1 b_3 u-\left[ \dfrac{b_3 c_2(a+ub_1)+b_2 c_3 \beta}{u\alpha}\right]\cos(\beta t)
    + \left[\dfrac{b_2 c_2\beta-b_3 c_3(a+ub_1)}{u\alpha}\right]\sin(\beta t) \\
 -c_1 b_2 u+\left[ \dfrac{b_2 c_2(a+ub_1)-b_3 c_3 \beta}{u\alpha}\right]\cos(\beta t)
    + \left[\dfrac{b_3 c_2\beta+b_2 c_3(a+ub_1)}{u\alpha}\right]\sin(\beta t) \\
c_1(a+ub_1)+c_2\cos(\beta t)+c_3\sin(\beta t) 
 \end{pmatrix}, 
\end{equation*}
with $t\in\R.$

\begin{obs}\label{solucion2} For controls $u\neq 0,$ it is possible to determine values for $c_1, c_2$ and $c_3$ such that $s(0)=s_0=(a_0,b_0,0)\in \mathbb{S}^2,$ for certain $a_0$ and $b_0.$ In fact, writing
\begin{equation*}
s(0)=\begin{pmatrix}
    c_1 b_3 u-\left[ \dfrac{b_3 c_2(a+ub_1)+b_2 c_3 \beta}{u\alpha}\right]\\
    -c_1 b_2 u+\left[ \dfrac{b_2 c_2(a+ub_1)-b_3 c_3 \beta}{u\alpha}\right]\\
    c_1(a+ub_1)+c_2
\end{pmatrix}    =\begin{pmatrix}
    a_0\\b_0\\0
\end{pmatrix},
\end{equation*}
we can take $a_0:=b_3/ \sqrt{\alpha}$ and $b_0:=-b_2/ \sqrt{\alpha},$ these satisfy $a_0b_2+b_0b_3=0$ and $a_0^2+b_0^2=1.$ Taking components we have
\begin{equation*}
    b_2 c_1 b_3 u-\dfrac{b_2 b_3 c_2 (a+ub_1)}{u\alpha}-\dfrac{b_2^2 c_3\beta }{u\alpha}=a_0 b_2 \hspace{0.2in}\mbox{and}\hspace{0.2in}-b_3 c_1 b_2  u+\dfrac{b_3 b_2 c_2 (a+ub_1)}{u\alpha}-\dfrac{b_3^2 c_3 \beta}{u\alpha}=b_0 b_3,  
\end{equation*}
therefore
\begin{equation*}
    -\dfrac{b_2^2 c_3\beta }{u\alpha}-\dfrac{b_3^2 c_3 \beta}{u\alpha}=-\dfrac{c_3\beta}{u}=0,
\end{equation*}
thus $c_3=0.$ On the other hand, replacing $c_2=-c_1(a+ub_1)$ in $$c_1 b_3 u-(b_3 c_2(a+ub_1))/(u\alpha)=a_0,$$ and using the definition of $\beta,$ we get
$c_1 b_3 \beta^2/u\alpha=a_0;$
therefore
\begin{equation*}
    c_1=\dfrac{a_0 u\alpha}{b_3 \beta^2} \hspace{0.3in}\mbox{and}\hspace{0.3in}
    c_2=-\dfrac{a_0 u\alpha(a+ub_1)}{b_3 \beta^2}.
\end{equation*}
Given $s_0=(a_0,b_0,0)\in \mathbb{S}^2;$ the trajectory of $\mathbb{P}\Sigma$, with control $u\neq 0$ and initial condition $s_0,$ is given by
\begin{equation*}
    s(t,s_0,u)=\begin{pmatrix}
        \dfrac{a_0 u^2\alpha}{\beta^2}+\dfrac{a_0(a+ub_1)^2}{\beta^2} \cos(\beta t)-\dfrac{a_0 b_2 (a+ub_1)}{b_3 \beta}\sin(\beta t)\\
        -\dfrac{a_0 b_2 u^2 \alpha}{b_3\beta^2}-\dfrac{a_0b_2 (a+ub_1)^2}{b_3 \beta^2}\cos(\beta t)-\dfrac{a_0(a+ub_1)}{\beta}\sin(\beta t)\\
        \dfrac{a_0 u \alpha(a+ub_1)}{b_3\beta^2}-\dfrac{a_0 u \alpha(a+ub_1)}{b_3 \beta^2}\cos(\beta t)
    \end{pmatrix}.
\end{equation*} Finally, writing $P_0^T=\left( \dfrac{a_0 u^2\alpha}{\beta^2}, -\dfrac{a_0 b_2 u \alpha}{b_3\beta^2}, \dfrac{a_0 u \alpha(a+ub_1)}{b_3\beta^2} \right),$ we have
\begin{equation*}
\|s(t;s_0,u)-P_0\| =\left|\dfrac{a_0(a+ub_1)}{b_3 \beta}\right|,
\end{equation*} \emph{i.e.,} the trajectory $s(t;s_0,u),$ with initial condition $s_0=(a_0,b_0,0)$ is a circle on the sphere $\mathbb{S}^2$ with center $P_0.$
\end{obs}

The main result of this paper is the following theorem.

\begin{teorema}\label{teocontrol}
The induced bilinear system $\mathbb{P}\Sigma: \dot{s}=As+uBs,\  s\in\mathbb{S}^2,$ given in \eqref{sistemajordan1}, is controllable on the sphere $\mathbb{S}^2$ if $[A,B]\neq 0.$
\end{teorema}
\begin{proof} We suppose that $\alpha=b_2^2+b_3^2\neq 0$ and consider the following cases:
\vspace{0.1in}

\noindent\textbf{Case 1.} If $\sqrt{\alpha}-b_1\neq 0;$ consider the control $u:=a/(\sqrt{\alpha}-b_1)$ such that $u^2\alpha =(a+ub_1)^2.$ Writing $s(t,s_0,u)=(s_1(t),s_2(t),s_3(t)),$ where $s_0=(a_0,b_0,0)$ is given in Observation \ref{solucion2}; taking $T:=\pi/ \beta,$ we easily get
\begin{equation*}
\begin{cases}
s_1(T)=\dfrac{a_0 u^2\alpha}{\beta^2}-\dfrac{a_0(a+ub_1)^2}{\beta^2}
=\dfrac{a_0 (a+ub_1)^2}{\beta^2}-\dfrac{a_0 (a+ub_1)^2}{\beta^2}
= 0; \\
s_2(T)=-\dfrac{a_0 b_2 u^2 \alpha}{b_3\beta^2}+\dfrac{a_0b_2 (a+ub_1)^2}{b_3 \beta^2}
     =-\dfrac{a_0 b_2 (a+ub_1)^2}{b_3\beta^2}+\dfrac{a_0b_2 (a+ub_1)^2}{b_3 \beta^2}
     =0; \\
s_3(T)
    =\dfrac{2a_0 u \alpha(a+ub_1)}{b_3 \beta^2}
    =2\dfrac{b_3}{\sqrt{\alpha}}\dfrac{a+ub_1}{\sqrt{\alpha}}\dfrac{\alpha(a+ub_1)}{b_3 \beta^2}
    =\dfrac{2\alpha u^2}{(a+ub_1)^2+u^2b_2^2+u^2b_3^2}=1;  
\end{cases}
\end{equation*}
\emph{i.e.,} there exist $T>0$ and $u$ such that $s(T;s_0,u)=(0,0,1).$ On the other hand, taking the initial condition $s_0$ with $a_0:=-b_3/\sqrt{\alpha}$ and $b_0:=b_2/\sqrt{\alpha},$ with the same control $u$ in $T:=\pi/ \beta,$ we get $s(T;s_0,u)=(0,0,-1).$ \vspace{0.1in}

\noindent\textbf{Case 2.} If $\sqrt{\alpha}-b_1=0;$ consider the control $ u:=-a/2\sqrt{\alpha}$ such that $u^2\alpha=(a+ub_1)^2.$ Therefore, taking $T:=\pi/\beta$ we get $s_1(T)=s_2(T)=0$ and
\begin{align*}
    s_3(T)&=\dfrac{2u \sqrt{\alpha}(a+u\sqrt{\alpha})}{\beta^2}
    =2\,\left(-\dfrac{a}{2\sqrt{\alpha}}\right)\dfrac{\sqrt{\alpha}(a+u\sqrt{\alpha})}{\beta^2}
    =-\dfrac{a(a+u\sqrt{\alpha})}{(a+u\sqrt{\alpha})^2+u^2\alpha}\\
          &=-\dfrac{a\left(a-\frac{a}{2}\right)}{\left(a-\frac{a}{2}\right)^2+\frac{a^2}{4}}=-1;
\end{align*}
\emph{i.e.,} there exist $T>0$ and $u$ such that $s(T;s_0,u)=(0,0,-1).$ Again, we can find an initial condition $s_0,$ with the same control $u$ and $T,$ such that $s(T;s_0,u)=(0,0,1).$ 

\noindent Let $\pi_3$ be the projection of the third coordinate of $\mathbb{R}^3.$ We obtain the controllability of $\mathbb{P}\Sigma$ as follows: given $\tilde{s_0}$
and $\tilde{s_1}$ in $\mathbb{S}^2,$ consider the following cases.
\vspace{0.1in}

\noindent \textbf{i.} $\pi_3(\tilde{s_0})=\pi_3(\tilde{s_1})=0:$ using Observation \ref{solucion}, consider the trajectory $s(t,\tilde{s_0},0),$ with initial condition $\tilde{s_0}$ and $u=0;$ since $\pi_3(\tilde{s_1})=0,$ for some $T>0,$  $s(T,\tilde{s_0},0)=\tilde{s_1}.$ See Figure 1. 

\begin{multicols}{2}

\tikzset{every picture/.style={line width=0.75pt}} 

\begin{tikzpicture}[x=0.75pt,y=0.75pt,yscale=-1,xscale=1]

\draw   (99,117.2) .. controls (99,62.41) and (143.41,18) .. (198.2,18) .. controls (252.98,18) and (297.4,62.41) .. (297.4,117.2) .. controls (297.4,171.98) and (252.98,216.4) .. (198.2,216.4) .. controls (143.41,216.4) and (99,171.98) .. (99,117.2) -- cycle ;
\draw  [dash pattern={on 0.84pt off 2.51pt}] (99.02,115.09) .. controls (99.27,104.04) and (143.87,96.08) .. (198.65,97.31) .. controls (253.42,98.53) and (297.62,108.48) .. (297.37,119.52) .. controls (297.13,130.56) and (252.52,138.52) .. (197.75,137.3) .. controls (142.98,136.07) and (98.78,126.13) .. (99.02,115.09) -- cycle ;
\draw    (242.11,135.56) ;
\draw [shift={(242.11,135.56)}, rotate = 0] [color={rgb, 255:red, 0; green, 0; blue, 0 }  ][fill={rgb, 255:red, 0; green, 0; blue, 0 }  ][line width=0.75]      (0, 0) circle [x radius= 3.35, y radius= 3.35]   ;
\draw    (146.11,133.56) ;
\draw [shift={(146.11,133.56)}, rotate = 0] [color={rgb, 255:red, 0; green, 0; blue, 0 }  ][fill={rgb, 255:red, 0; green, 0; blue, 0 }  ][line width=0.75]      (0, 0) circle [x radius= 3.35, y radius= 3.35]   ;
\draw   (189.56,132.59) -- (200.1,137.82) -- (188.95,141.57) ;
\draw  [draw opacity=0] (242.69,136.14) .. controls (229.18,137.22) and (213.91,137.66) .. (197.75,137.3) .. controls (178.85,136.88) and (161.2,135.41) .. (146.21,133.23) -- (198.2,117.3) -- cycle ; \draw   (242.69,136.14) .. controls (229.18,137.22) and (213.91,137.66) .. (197.75,137.3) .. controls (178.85,136.88) and (161.2,135.41) .. (146.21,133.23) ;  

\draw (135.4,144.39) node [anchor=north west][inner sep=0.75pt]  [rotate=-356.45,xslant=0.06]  {$\widetilde{s_{0}}$};
\draw (236,142.4) node [anchor=north west][inner sep=0.75pt]    {$\widetilde{s_{1}} =s\left( T,\widetilde{s_{0}} ,0\right)$};
\draw (173,223) node [anchor=north west][inner sep=0.75pt]   [align=left] {Figure 1};
\end{tikzpicture}

\tikzset{every picture/.style={line width=0.75pt}} 
\begin{tikzpicture}[x=0.75pt,y=0.75pt,yscale=-1,xscale=1]

\draw   (119,137.3) .. controls (119,82.52) and (163.41,38.1) .. (218.2,38.1) .. controls (272.98,38.1) and (317.4,82.52) .. (317.4,137.3) .. controls (317.4,192.09) and (272.98,236.5) .. (218.2,236.5) .. controls (163.41,236.5) and (119,192.09) .. (119,137.3) -- cycle ;
\draw  [dash pattern={on 0.84pt off 2.51pt}] (119.02,135.09) .. controls (119.27,124.04) and (163.87,116.08) .. (218.65,117.31) .. controls (273.42,118.53) and (317.62,128.48) .. (317.37,139.52) .. controls (317.13,150.56) and (272.52,158.52) .. (217.75,157.3) .. controls (162.98,156.07) and (118.78,146.13) .. (119.02,135.09) -- cycle ;
\draw    (171.11,81.56) ;
\draw [shift={(171.11,81.56)}, rotate = 0] [color={rgb, 255:red, 0; green, 0; blue, 0 }  ][fill={rgb, 255:red, 0; green, 0; blue, 0 }  ][line width=0.75]      (0, 0) circle [x radius= 3.35, y radius= 3.35]   ;
\draw    (166.11,153.56) ;
\draw [shift={(166.11,153.56)}, rotate = 0] [color={rgb, 255:red, 0; green, 0; blue, 0 }  ][fill={rgb, 255:red, 0; green, 0; blue, 0 }  ][line width=0.75]      (0, 0) circle [x radius= 3.35, y radius= 3.35]   ;
\draw  [dash pattern={on 0.84pt off 2.51pt}] (147.13,70.35) .. controls (147.27,63.77) and (180.05,59.17) .. (220.34,60.07) .. controls (260.63,60.97) and (293.18,67.03) .. (293.03,73.61) .. controls (292.88,80.18) and (260.1,84.78) .. (219.81,83.88) .. controls (179.52,82.98) and (146.98,76.92) .. (147.13,70.35) -- cycle ;
\draw    (236.11,59.56) ;
\draw [shift={(236.11,59.56)}, rotate = 0] [color={rgb, 255:red, 0; green, 0; blue, 0 }  ][fill={rgb, 255:red, 0; green, 0; blue, 0 }  ][line width=0.75]      (0, 0) circle [x radius= 3.35, y radius= 3.35]   ;
\draw    (247.11,156.56) ;
\draw [shift={(247.11,156.56)}, rotate = 0] [color={rgb, 255:red, 0; green, 0; blue, 0 }  ][fill={rgb, 255:red, 0; green, 0; blue, 0 }  ][line width=0.75]      (0, 0) circle [x radius= 3.35, y radius= 3.35]   ;
\draw  [draw opacity=0] (171.7,79.84) .. controls (156.56,77.32) and (147.05,73.91) .. (147.13,70.35) .. controls (147.27,63.77) and (180.05,59.17) .. (220.34,60.07) .. controls (225.76,60.19) and (231.04,60.41) .. (236.11,60.7) -- (220.08,71.98) -- cycle ; \draw   (171.7,79.84) .. controls (156.56,77.32) and (147.05,73.91) .. (147.13,70.35) .. controls (147.27,63.77) and (180.05,59.17) .. (220.34,60.07) .. controls (225.76,60.19) and (231.04,60.41) .. (236.11,60.7) ;  
\draw  [draw opacity=0] (247.2,157.06) .. controls (237.89,157.43) and (228,157.53) .. (217.75,157.3) .. controls (198.97,156.88) and (181.42,155.43) .. (166.49,153.27) -- (218.2,137.3) -- cycle ; \draw   (247.2,157.06) .. controls (237.89,157.43) and (228,157.53) .. (217.75,157.3) .. controls (198.97,156.88) and (181.42,155.43) .. (166.49,153.27) ;  
\draw  [dash pattern={on 0.84pt off 2.51pt}] (215.5,39.22) .. controls (224.22,35.62) and (247.01,70.74) .. (266.39,117.67) .. controls (285.77,164.61) and (294.41,205.57) .. (285.69,209.17) .. controls (276.97,212.77) and (254.19,177.65) .. (234.81,130.72) .. controls (215.43,83.78) and (206.78,42.82) .. (215.5,39.22) -- cycle ;
\draw  [draw opacity=0] (236.71,59.13) .. controls (245.97,73.34) and (256.61,94) .. (266.39,117.67) .. controls (285.77,164.61) and (294.41,205.57) .. (285.69,209.17) .. controls (278.69,212.07) and (262.61,189.97) .. (246.51,156.84) -- (250.6,124.19) -- cycle ; \draw   (236.71,59.13) .. controls (245.97,73.34) and (256.61,94) .. (266.39,117.67) .. controls (285.77,164.61) and (294.41,205.57) .. (285.69,209.17) .. controls (278.69,212.07) and (262.61,189.97) .. (246.51,156.84) ;  
\draw   (257.09,110.16) -- (257.91,98.33) -- (265.54,107.41) ;
\draw   (270.13,189.5) -- (272.75,201.07) -- (262.82,194.57) ;
\draw   (194.15,64.95) -- (183.11,60.61) -- (194.07,56.06) ;
\draw   (199.75,151.4) -- (210.08,157.03) -- (198.79,160.35) ;

\draw (155.4,164.39) node [anchor=north west][inner sep=0.75pt]  [rotate=-356.45,xslant=0.06]  {$\widetilde{s_{0}}$};
\draw (187,167.4) node [anchor=north west][inner sep=0.75pt]    {$s\left( t,\widetilde{s_{0}} ,0\right)$};
\draw (187,240) node [anchor=north west][inner sep=0.75pt]   [align=left] {Figure 2};
\draw (148.13,80.63) node [anchor=north west][inner sep=0.75pt]  [rotate=-356.45,xslant=0.06]  {$\widetilde{s_{1}}$};
\draw (222.2,135.59) node [anchor=north west][inner sep=0.75pt]  [rotate=-356.45,xslant=0.06]  {$P_{0}$};
\draw (235.4,41.39) node [anchor=north west][inner sep=0.75pt]  [rotate=-356.45,xslant=0.06]  {$P_{1}$};
\draw (269,96.4) node [anchor=north west][inner sep=0.75pt]    {$s( t,P_{0} ,u)$};
\draw (81,53.4) node [anchor=north west][inner sep=0.75pt]    {$s( t,P_{1} ,0)$};
\end{tikzpicture}
\end{multicols}

\noindent \textbf{ii.} $\pi_3(\tilde{s_0})=0$ and $\pi_3(\tilde{s_1})>0:$ consider the trajectory $s(t,\tilde{s_0},0),$ with initial condition $\tilde{s_0}$ and control $u=0;$ there exists $T_1>0$ such that $s(T_1,\tilde{s_0},0)=P_0,$ where $P_0$ is such that for some $u\neq 0$ and $T>0$ satisfies $s(T,P_0,u)=(0,0,1).$ The trajectory $s(t,P_0,u),$ at some time $T_2$ intersects, at a point $P_1,$ the trajectory $s(t, \tilde{s_1},0),$ with initial condition $ \tilde{s_1}$ and control $u=0;$ finally, at some time $T_3>0,$ starting from $P_1,$ we have that $s(T_3,P_1,0)=\tilde{s_1}.$ See Figure 2.

\vspace{0.1in}

\noindent \textbf{iii.} $\pi_3(\tilde{s_0})<0$ and $\pi_3(\tilde{s_1})>0:$ consider $P_1$ and $P_2,$ and the controls $u_1\neq 0$ and $u_2\neq 0$ such that, at some times $t_1$ and $t_2,$ we have that $s(t_1,P_1,u_1)=(0,0,-1)$ and $s(t_2,P_2,u_2)=(0,0,1).$ The trajectory $s(t,\tilde{s_0},0)$ intersects $s(t,P_1,u_1)$ at a point $P_0$ and at some $T_0;$ thus, the exists a time $T_1$ such that $s(T_1,P_0,u_1)=P_1;$ therefore, there exists $T_2$ such that $s(T_2,P_1,0)=P_2.$ The trajectory $s(t,P_2,u_2)$ find the trajectory $s(t,\tilde{s_1},0)$ at a point $P_3$ at time $T_3;$ finally, there exists $T_4$ such that $s(T_4,P_3,0)=\tilde{s_1}.$ See Figure 3.

\vspace{0.2in}

\begin{center}

\tikzset{every picture/.style={line width=0.75pt}} 

\begin{tikzpicture}[x=0.75pt,y=0.75pt,yscale=-1,xscale=1]

\draw   (139,139.19) .. controls (139,84.41) and (183.41,39.99) .. (238.2,39.99) .. controls (292.98,39.99) and (337.4,84.41) .. (337.4,139.19) .. controls (337.4,193.98) and (292.98,238.39) .. (238.2,238.39) .. controls (183.41,238.39) and (139,193.98) .. (139,139.19) -- cycle ;
\draw  [dash pattern={on 0.84pt off 2.51pt}] (139.02,136.98) .. controls (139.27,125.93) and (183.87,117.97) .. (238.65,119.2) .. controls (293.42,120.42) and (337.62,130.36) .. (337.37,141.41) .. controls (337.13,152.45) and (292.52,160.41) .. (237.75,159.19) .. controls (182.98,157.96) and (138.78,148.02) .. (139.02,136.98) -- cycle ;
\draw    (191.11,83.44) ;
\draw [shift={(191.11,83.44)}, rotate = 0] [color={rgb, 255:red, 0; green, 0; blue, 0 }  ][fill={rgb, 255:red, 0; green, 0; blue, 0 }  ][line width=0.75]      (0, 0) circle [x radius= 3.35, y radius= 3.35]   ;
\draw    (181.29,121.59) ;
\draw [shift={(181.29,121.59)}, rotate = 0] [color={rgb, 255:red, 0; green, 0; blue, 0 }  ][fill={rgb, 255:red, 0; green, 0; blue, 0 }  ][line width=0.75]      (0, 0) circle [x radius= 3.35, y radius= 3.35]   ;
\draw  [dash pattern={on 0.84pt off 2.51pt}] (167.13,72.24) .. controls (167.27,65.66) and (200.05,61.06) .. (240.34,61.96) .. controls (280.63,62.86) and (313.18,68.92) .. (313.03,75.5) .. controls (312.88,82.07) and (280.1,86.67) .. (239.81,85.77) .. controls (199.52,84.87) and (166.98,78.81) .. (167.13,72.24) -- cycle ;
\draw    (256.11,61.44) ;
\draw [shift={(256.11,61.44)}, rotate = 0] [color={rgb, 255:red, 0; green, 0; blue, 0 }  ][fill={rgb, 255:red, 0; green, 0; blue, 0 }  ][line width=0.75]      (0, 0) circle [x radius= 3.35, y radius= 3.35]   ;
\draw    (267.11,158.44) ;
\draw [shift={(267.11,158.44)}, rotate = 0] [color={rgb, 255:red, 0; green, 0; blue, 0 }  ][fill={rgb, 255:red, 0; green, 0; blue, 0 }  ][line width=0.75]      (0, 0) circle [x radius= 3.35, y radius= 3.35]   ;
\draw  [draw opacity=0] (191.7,81.73) .. controls (176.56,79.21) and (167.05,75.8) .. (167.13,72.24) .. controls (167.27,65.66) and (200.05,61.06) .. (240.34,61.96) .. controls (245.76,62.08) and (251.04,62.3) .. (256.11,62.59) -- (240.08,73.87) -- cycle ; \draw   (191.7,81.73) .. controls (176.56,79.21) and (167.05,75.8) .. (167.13,72.24) .. controls (167.27,65.66) and (200.05,61.06) .. (240.34,61.96) .. controls (245.76,62.08) and (251.04,62.3) .. (256.11,62.59) ;  
\draw  [draw opacity=0] (267.2,158.95) .. controls (257.89,159.32) and (248,159.42) .. (237.75,159.19) .. controls (182.98,157.96) and (138.78,148.02) .. (139.02,136.98) .. controls (139.18,130.23) and (155.86,124.64) .. (181.29,121.59) -- (238.2,139.19) -- cycle ; \draw   (267.2,158.95) .. controls (257.89,159.32) and (248,159.42) .. (237.75,159.19) .. controls (182.98,157.96) and (138.78,148.02) .. (139.02,136.98) .. controls (139.18,130.23) and (155.86,124.64) .. (181.29,121.59) ;  
\draw  [dash pattern={on 0.84pt off 2.51pt}] (235.5,41.11) .. controls (244.22,37.5) and (267.01,72.63) .. (286.39,119.56) .. controls (305.77,166.49) and (314.41,207.46) .. (305.69,211.06) .. controls (296.97,214.66) and (274.19,179.54) .. (254.81,132.6) .. controls (235.43,85.67) and (226.78,44.71) .. (235.5,41.11) -- cycle ;
\draw  [draw opacity=0] (256.71,61.01) .. controls (265.97,75.22) and (276.61,95.89) .. (286.39,119.56) .. controls (305.77,166.49) and (314.41,207.46) .. (305.69,211.06) .. controls (298.69,213.95) and (282.61,191.86) .. (266.51,158.73) -- (270.6,126.08) -- cycle ; \draw   (256.71,61.01) .. controls (265.97,75.22) and (276.61,95.89) .. (286.39,119.56) .. controls (305.77,166.49) and (314.41,207.46) .. (305.69,211.06) .. controls (298.69,213.95) and (282.61,191.86) .. (266.51,158.73) ;  
\draw   (277.09,112.05) -- (277.91,100.22) -- (285.54,109.29) ;
\draw   (290.13,191.39) -- (292.75,202.96) -- (282.82,196.46) ;
\draw   (214.15,66.83) -- (203.11,62.5) -- (214.07,57.95) ;
\draw   (219.75,153.29) -- (230.08,158.92) -- (218.79,162.24) ;
\draw  [dash pattern={on 0.84pt off 2.51pt}] (174.01,213.89) .. controls (174.13,208.37) and (201.85,204.52) .. (235.91,205.28) .. controls (269.98,206.04) and (297.49,211.13) .. (297.37,216.64) .. controls (297.25,222.16) and (269.53,226.01) .. (235.47,225.25) .. controls (201.4,224.49) and (173.88,219.4) .. (174.01,213.89) -- cycle ;
\draw  [dash pattern={on 0.84pt off 2.51pt}] (144.81,105.61) .. controls (151.36,93.75) and (178.03,113.06) .. (204.36,148.74) .. controls (230.69,184.42) and (246.73,222.96) .. (240.17,234.82) .. controls (233.62,246.68) and (206.96,227.37) .. (180.62,191.69) .. controls (154.29,156.01) and (138.25,117.47) .. (144.81,105.61) -- cycle ;
\draw    (235.91,205.28) ;
\draw [shift={(235.91,205.28)}, rotate = 0] [color={rgb, 255:red, 0; green, 0; blue, 0 }  ][fill={rgb, 255:red, 0; green, 0; blue, 0 }  ][line width=0.75]      (0, 0) circle [x radius= 3.35, y radius= 3.35]   ;
\draw    (209.11,222.56) ;
\draw [shift={(209.11,222.56)}, rotate = 0] [color={rgb, 255:red, 0; green, 0; blue, 0 }  ][fill={rgb, 255:red, 0; green, 0; blue, 0 }  ][line width=0.75]      (0, 0) circle [x radius= 3.35, y radius= 3.35]   ;
\draw  [draw opacity=0] (182.17,122.5) .. controls (189.28,129.68) and (196.83,138.54) .. (204.36,148.74) .. controls (219.21,168.86) and (230.78,189.89) .. (236.88,206.53) -- (192.49,170.21) -- cycle ; \draw   (182.17,122.5) .. controls (189.28,129.68) and (196.83,138.54) .. (204.36,148.74) .. controls (219.21,168.86) and (230.78,189.89) .. (236.88,206.53) ;  
\draw  [draw opacity=0] (236.38,205.29) .. controls (270.23,206.09) and (297.49,211.15) .. (297.37,216.64) .. controls (297.25,222.16) and (269.53,226.01) .. (235.47,225.25) .. controls (225.6,225.03) and (216.29,224.45) .. (208.03,223.59) -- (235.69,215.26) -- cycle ; \draw   (236.38,205.29) .. controls (270.23,206.09) and (297.49,211.15) .. (297.37,216.64) .. controls (297.25,222.16) and (269.53,226.01) .. (235.47,225.25) .. controls (225.6,225.03) and (216.29,224.45) .. (208.03,223.59) ;  
\draw   (246.37,220.69) -- (257.11,225.49) -- (246.11,229.69) ;
\draw   (220.8,182.34) -- (218.76,170.75) -- (228.39,177.51) ;
\draw   (164.34,130.58) -- (153.12,126.74) -- (163.86,121.71) ;

\draw (189.4,194.28) node [anchor=north west][inner sep=0.75pt]  [rotate=-356.45,xslant=0.06]  {$\widetilde{s_{0}}$};
\draw (342,134.29) node [anchor=north west][inner sep=0.75pt]    {$s( t,P_{1} ,0)$};
\draw (209,246.89) node [anchor=north west][inner sep=0.75pt]   [align=left] {Figure 3};
\draw (197.13,72.52) node [anchor=north west][inner sep=0.75pt]  [rotate=-356.45,xslant=0.06]  {$\widetilde{s_{1}}$};
\draw (242.2,137.47) node [anchor=north west][inner sep=0.75pt]  [rotate=-356.45,xslant=0.06]  {$P_{2}$};
\draw (255.4,43.28) node [anchor=north west][inner sep=0.75pt]  [rotate=-356.45,xslant=0.06]  {$P_{3}$};
\draw (291,107.29) node [anchor=north west][inner sep=0.75pt]    {$s( t,P_{2} ,u_{2})$};
\draw (95,53.29) node [anchor=north west][inner sep=0.75pt]    {$s( t,P_{3} ,0)$};
\draw (180.2,98.47) node [anchor=north west][inner sep=0.75pt]  [rotate=-356.45,xslant=0.06]  {$P_{1}$};
\draw (236.2,183.47) node [anchor=north west][inner sep=0.75pt]  [rotate=-356.45,xslant=0.06]  {$P_{0}$};
\draw (68,93.29) node [anchor=north west][inner sep=0.75pt]    {$s( t,P_{0} ,u_{1})$};
\draw (103,202.29) node [anchor=north west][inner sep=0.75pt]    {$s\left( t,\widetilde{s_{0}} ,0\right)$};
\end{tikzpicture}
\end{center}
At every case, there exists a time $T^*>0$ and a control $u^*$ such that $s(T^*,\tilde{s_0},u^*)=\tilde{s_1}.$
\end{proof}

In the following example we make explicit the technique used to ensure the controllability of bilinear systems induced on the sphere $\mathbb{S}^2.$

\begin{ejemplo}
We consider the following induced bilinear system 
\begin{equation*}
\dot{s}=As+uBs=\begin{pmatrix}
            0&1&0\\ -1 &0&0\\0&0&0
        \end{pmatrix}s+u \begin{pmatrix}
            0&0&1\\ 0 &0&1\\-1&-1&0
        \end{pmatrix}  s, \hspace{0.3in} s\in \mathbb{S}^2.
\end{equation*}
The matrix $A+uB=\begin{pmatrix}
         0&1&u\\ -1 &0&u\\-u&-u&0
        \end{pmatrix}$ 
has eigenvalues $\lambda_1=0$ and $\lambda_{2,3}=\pm i\beta,$ where $\beta =\sqrt{1+2u^2}.$ According to Observations \ref{solucion} and \ref{solucion2}, for $u=0$ and initial condition $s_0=(a,b,c)\in \mathbb{S}^2,$ the trajectory of the system is given by
\begin{equation*}
        s(t,s_0,0)=\begin{pmatrix}
        b\sin(t)+a\cos(t)\\ -a\sin(t)+b\cos(t)\\c
\end{pmatrix};
\end{equation*}
on the other hand, for $u\neq 0$ and initial condition  $s_0=(0,1,0),$ the trajectories of the system are given by
\begin{equation*}
 s(t,s_0,u)=\begin{pmatrix}
        -\dfrac{u^2}{\beta^2}+\dfrac{u^2}{\beta^2}\cos(\beta t)+\dfrac{1}{\beta}\sin(\beta t)\\
        \dfrac{u^2}{\beta^2}+\dfrac{1+u^2}{\beta^2}\cos(\beta t)\\
        -\dfrac{u}{\beta^2}+\dfrac{u}{\beta^2}\cos(\beta t)-\dfrac{u}{\beta}\sin(\beta t)
\end{pmatrix}.
\end{equation*}
Taking the control $u=1$ we have $\beta =\sqrt{3},$ thus, in $T=2\pi/3\sqrt{3},$ since $\cos(\sqrt{3} T)=-1/2$ and $\sin(\sqrt{3} T)=\sqrt{3}/2,$ we get 
\begin{equation*}
s(T,s_0,1)=\begin{pmatrix}
            0\\0\\ -1
    \end{pmatrix};
\end{equation*}
we note that, for $P_0^T:=(
    -1/3, 1/3, -1/3)
$ we get $\|s(t,s_0,1)-P_0\|^2=2/3,$ \emph{i.e.,} the trajectory is a circle centered at $P_0$ that passes through $s_0=(0,1,0)$ and $s(T,s_0,1)=(0,0,-1).$ Analogously, taking $u=-1$ and $T=2\pi/3\sqrt{3},$ we get
\begin{equation*}
s(T,s_0,-1)=\begin{pmatrix}
            0\\0\\ 1
\end{pmatrix};
\end{equation*}
for $P_1^T:=(-1/3, 1/3, 1/3)$ we have $\|s(t,s_0,-1)-P_1\|^2=2/3,$ \emph{i.e.,} the trajectory is a circle centered at $P_1$ that passes through $s_0=(0,1,0)$ and $s(T,s_0,-1)=(0,0,1).$ See Figure 4.
\begin{center}

\tikzset{every picture/.style={line width=0.75pt}} 

\begin{tikzpicture}[x=0.75pt,y=0.75pt,yscale=-1,xscale=1]

\draw   (139,139.19) .. controls (139,84.41) and (183.41,39.99) .. (238.2,39.99) .. controls (292.98,39.99) and (337.4,84.41) .. (337.4,139.19) .. controls (337.4,193.98) and (292.98,238.39) .. (238.2,238.39) .. controls (183.41,238.39) and (139,193.98) .. (139,139.19) -- cycle ;
\draw    (238.2,39.99) ;
\draw [shift={(238.2,39.99)}, rotate = 0] [color={rgb, 255:red, 0; green, 0; blue, 0 }  ][fill={rgb, 255:red, 0; green, 0; blue, 0 }  ][line width=0.75]      (0, 0) circle [x radius= 3.35, y radius= 3.35]   ;
\draw    (288.52,90) ;
\draw [shift={(288.52,90)}, rotate = 0] [color={rgb, 255:red, 0; green, 0; blue, 0 }  ][fill={rgb, 255:red, 0; green, 0; blue, 0 }  ][line width=0.75]      (0, 0) circle [x radius= 3.35, y radius= 3.35]   ;
\draw   (239.67,38.59) .. controls (244.44,34.06) and (270.17,53.41) .. (297.15,81.8) .. controls (324.13,110.19) and (342.14,136.88) .. (337.37,141.41) .. controls (332.61,145.94) and (306.87,126.59) .. (279.89,98.2) .. controls (252.91,69.81) and (234.9,43.12) .. (239.67,38.59) -- cycle ;
\draw    (337.37,141.41) ;
\draw [shift={(337.37,141.41)}, rotate = 0] [color={rgb, 255:red, 0; green, 0; blue, 0 }  ][fill={rgb, 255:red, 0; green, 0; blue, 0 }  ][line width=0.75]      (0, 0) circle [x radius= 3.35, y radius= 3.35]   ;
\draw    (286.81,188.89) ;
\draw [shift={(286.81,188.89)}, rotate = 0] [color={rgb, 255:red, 0; green, 0; blue, 0 }  ][fill={rgb, 255:red, 0; green, 0; blue, 0 }  ][line width=0.75]      (0, 0) circle [x radius= 3.35, y radius= 3.35]   ;
\draw   (307.7,100.51) -- (304.1,89.21) -- (314.54,94.84) ;
\draw   (292.56,193.92) -- (303.65,189.71) -- (298.59,200.44) ;
\draw   (281.94,94.4) -- (286.46,105.37) -- (275.59,100.62) ;
\draw   (288.72,177.66) -- (277.34,180.62) -- (283.3,170.47) ;
\draw    (238.2,238.39) ;
\draw [shift={(238.2,238.39)}, rotate = 0] [color={rgb, 255:red, 0; green, 0; blue, 0 }  ][fill={rgb, 255:red, 0; green, 0; blue, 0 }  ][line width=0.75]      (0, 0) circle [x radius= 3.35, y radius= 3.35]   ;
\draw   (236.22,238.59) .. controls (231.61,233.89) and (250.52,207.84) .. (278.46,180.39) .. controls (306.4,152.95) and (332.79,134.5) .. (337.4,139.19) .. controls (342.01,143.88) and (323.09,169.94) .. (295.15,197.38) .. controls (267.21,224.83) and (240.83,243.28) .. (236.22,238.59) -- cycle ;
\draw  [dash pattern={on 0.84pt off 2.51pt}] (139.47,141.64) .. controls (139.47,133.48) and (183.77,126.87) .. (238.42,126.87) .. controls (293.07,126.87) and (337.37,133.48) .. (337.37,141.64) .. controls (337.37,149.79) and (293.07,156.41) .. (238.42,156.41) .. controls (183.77,156.41) and (139.47,149.79) .. (139.47,141.64) -- cycle ;

\draw (197,170.29) node [anchor=north west][inner sep=0.75pt]    {$s( t,s_{0} ,1)$};
\draw (213,271.89) node [anchor=north west][inner sep=0.75pt]   [align=left] {Figure 4};
\draw (267.2,189.47) node [anchor=north west][inner sep=0.75pt]  [rotate=-356.45,xslant=0.06]  {$P_{0}$};
\draw (267.4,67.28) node [anchor=north west][inner sep=0.75pt]  [rotate=-356.45,xslant=0.06]  {$P_{1}$};
\draw (186,87.29) node [anchor=north west][inner sep=0.75pt]    {$s( t,s_{0} ,-1)$};
\draw (346,130.29) node [anchor=north west][inner sep=0.75pt]    {$s_{0} =( 0,1,0)$};
\draw (163,16.29) node [anchor=north west][inner sep=0.75pt]    {$( 0,0,1) =s( T,s_{0} ,-1)$};
\draw (156,243.4) node [anchor=north west][inner sep=0.75pt]    {$( 0,0,-1) =s( T,s_{0} ,1)$};
\end{tikzpicture}
\end{center}
Using these trajectories, with controls $u=\pm 1,$ and the circular trajectories, with control $u=0,$ we obtain the controllability of the system.
\end{ejemplo}

\paragraph{\bf Acknowledgments.} This work is part of the first author’s master thesis \cite{MC}; he thanks the Terminal Master's program in Mathematics at the Universidad Mayor de San Andrés.

\Addresses

\end{document}